\documentclass[11pt,reqno]{amsart}

\date{\today}

\usepackage[ansinew]{inputenc}
\usepackage{textcomp}
\usepackage{cite}

\usepackage{amsthm}
\usepackage{amssymb}
\usepackage{amsfonts}

\usepackage{enumerate}

\newtheorem{theorem}{Theorem}

\newtheorem{lemma}[theorem]{Lemma}

\theoremstyle{definition}

\theoremstyle{definition}

\newcommand{\lk}{\left(}
\newcommand{\ls}{\left<}

\newcommand{\N}{\mathbb{N}}

\newcommand{\R}{\mathbb{R}}
\newcommand{\rk}{\right)}
\newcommand{\rs}{\right>}

\newcommand{\tr}{\textnormal{Tr}}

\begin{document}

\title[A short proof of Weyl's law]{A short proof of Weyl's law for fractional differential operators}

\author[Leander Geisinger]{Leander Geisinger}
 \address{Leander Geisinger, Department of Physics, Princeton University, Princeton, NJ 08544, USA}
\email{leander@princeton.edu}

\thanks{\copyright\, 2013 by the author. This paper may be reproduced, in its entirety, for non-commercial purposes.}

\begin{abstract}
We study spectral asymptotics for a large class of differential operators on an open subset of $\R^d$ with finite volume. This class includes the Dirichlet Laplacian, the fractional Laplacian, and also fractional differential operators with non-homogeneous symbols. Based on a sharp estimate for the sum of the eigenvalues  we establish  the first term of the semiclassical asymptotics. This generalizes Weyl's law for the Laplace operator. 
\end{abstract}

\maketitle

\section{Introduction and main result}

In this note we study the asymptotic distribution of eigenvalues of a differential operator $\mathcal L_T$ with symbol $T  :  \R^d \to \R$ on an open set $\Omega \subset \R^d$ of finite volume.
We define the operator $\mathcal L_T$ in terms of the quadratic form
$$
Q_{T}[u] = \int_{\R^d}  T(p) |\hat u(p)|^2 dp
$$
restricted to the form domain
$$
\mathcal H_T(\Omega) = \left\{ u \in L^2(\R^d) \, : \, \int_{\R^d} \lk 1 + T(p)  \rk | \hat u(p)|^2 dp < \infty \ \textnormal{and} \ u \equiv 0 \ \textnormal{on} \ \R^d \setminus \Omega \right\} \, ,
$$
where $\hat u(p) = (2\pi)^{-d/2} \int e^{-ip\cdot x} u(x) dx$ is the Fourier transform of $u$. Then $\mathcal L_T$ is defined to be the self-adjoint operator satisfying $Q_T[u] = \ls u,\mathcal L_T u \rs$, where $\ls \cdot,\cdot \rs$ denotes the scalar product in $L^2(\R^d)$.

For example, for $T(p) = | p |^2$ we obtain the Dirichlet Laplace operator, with form domain given by the Sobolev space $H_0^1(\Omega)$. More generally, for $T(p) = |p|^{2s}$, $0 < s \leq 1$, $\mathcal L_T$ is the fractional Laplacian $(-\Delta)^s$, see for example \cite{DinPalVal12} for details concerning the definition.

For the Dirichlet Laplacian spectral asymptotics are well-known: For $\Omega \subset \R^d$ of finite volume the spectrum is discrete  and consists of eigenvalues $0 < \lambda_1 \leq \lambda_2 \leq \dots$ with finite multiplicities. 
In 1912, Weyl proved the famous asymptotic law \cite{Wey12a}
$$
\lambda_n =  \frac{4 \pi \Gamma(1+d/2)^{2/d}}{|\Omega|^{2/d}} n^{2/d} \lk 1+ o(1) \rk \, , \qquad n \to \infty \, ,
$$
initially for a bounded domain in two dimensions. Here $|\Omega|$ denotes the volume of $\Omega$  and $\Gamma$ is the gamma function. This result was later improved and generalized in various ways, see for example \cite{BirSol80,Hoe85,SafVas97,Ivr98} for a summary of results and applications concerning the  Laplace operator and further references.

The purpose of this note is to give a short proof of the analogue of this result for the operator $\mathcal L_T$ under minimal assumptions on the set $\Omega$ (we only require that the volume is finite) and under weak assumptions on the symbol $T$. The method is elementary, in particular, we do not use bracketing as in \cite[Thm. XIII.78]{ReeSim78} and we need only a simple version of a Tauberian theorem, see Lemma \ref{lem:tauber} below. Along the way, in Theorem \ref{thm:ber}, we prove a sharp bound on the sum of the eigenvalues of $\mathcal L_T$, an analogue of the Berezin-Li-Yau inequality  for the Dirichlet Laplacian \cite{Ber72a,LiYau83}.

We assume that $T$ satisfies the following conditions:

\begin{enumerate}[I]
\item \label{a1}
There is a function $T_0 :  \R^d \to \R$ with the following three properties. $T_0$ is homogeneous of degree $\alpha > 0$: $T_0(\nu p) = \nu^{\alpha} T_0(p)$ for $p \in \R^d$ and $\nu > 0$. The set of $p \in \R^d$ with $T_0(p) < 1$ has finite Lebesgue-measure:
$$
V_T := \left| \left\{ p \in \R^d \, : \, T_0(p) < 1 \right\} \right| < \infty \, .
$$
The function $T_0$ approximates $T$ in the sense that for all $p \in \R^d$
$$
\lim_{\nu \to \infty} \nu^{-\alpha} T(\nu p) = T_0(p) \, .
$$
Moreover, if this convergence is not uniform in $p \in \R^d$ we also require that there is a function $\tilde T  :  \R^d  \to \R$ which is locally integrable, larger than $1$ at infinity (i.e. there is $R > 0$ such that $\tilde T(p) \geq 1$ for $|p| \geq R$), and such that there is $\nu_0 > 0$ with $\tilde T(p) \leq \nu^{-\alpha} T(\nu p)$ for all $p \in \R^d$ and $\nu \geq \nu_0$.

\item \label{a2}
There are constants $C_0 >0$ and  $N \in \N$ such that for all $\eta \in \R^d$
$$
\sup_{p \in \R^d} \lk \frac 12 \lk T(p+\eta) + T(p-\eta) \rk - T(p) \rk \leq C_0 (1+|\eta|)^N \, .
$$
\end{enumerate}

Under these assumptions, since the volume of $\Omega$ is finite, the spectrum of $\mathcal L_T$ is discrete and consists of positive eigenvalues with finite multiplicities. We denote these eigenvalues by $(\lambda_k)_{k \in \N}$. To study the asymptotic distribution we introduce the counting function $N(\Lambda)$ that counts the number of eigenvalues below $\Lambda > 0$:
$$
N(\Lambda) = \sum_{k \in \N} \lk \Lambda - \lambda_k \rk_+^0 = \tr \lk \mathcal L_T - \Lambda \rk_-^0 \, .
$$
Here $x_\pm = (|x|\pm x)/2$ denotes the positive and negative part of  $x \in \R$ and we use the same convention for self-adjoint operators.

\begin{theorem}
\label{thm:weyl}
Let $\Omega \subset \R^d$ be an open set of finite volume and assume that the symbol $T$ satisfies assumptions \ref{a1} and \ref{a2}. Then,  as $\Lambda \to \infty$, the asymptotic formula 
\begin{equation}
\label{weyl}
\lim_{\Lambda \to \infty} \Lambda^{-d/\alpha} N(\Lambda) = (2\pi)^{-d} \, |\Omega| \, V_T 
\end{equation}
holds.
\end{theorem}

The assumptions \ref{a1} and \ref{a2} are satisfied by the Dirichlet Laplacian and more generally by the fractional Laplacian $(-\Delta)^s$ with symbol $T(p) = |p|^{2s}$, $0 < s \leq 1$.  In this case the symbol itself is homogeneous and we have $T = T_0$ in assumption \ref{a1}. For the fractional Laplacian Theorem~\ref{thm:weyl} was proved by Blumenthal and Getoor  under additional assumptions on the boundary of $\Omega$ \cite{BluGet59}. 

We emphasize that the theorem also applies to non-homogeneous symbols, for example to the operator with symbol $T(p) = |p|^\alpha \pm |p|^\beta$, $0< \beta < \alpha \leq 2$. In this case we have $T_0(p) = |p|^\alpha$ and we can choose $\tilde T(p) = T_0(p)$ for $T(p) = |p|^\alpha+ |p|^\beta$ and $\tilde T(p) = T(p)$ for $T(p) = |p|^\alpha - |p|^\beta$.

To illustrate the result let us consider the fractional differential operator 
$$
\mathcal L_{2s}= \sum_{i=1}^d (-\partial_i^2)^s \, , \qquad 0 < s \leq 1 \, ,
$$
that was recently studied in \cite{Hat13}.  This operator corresponds to $\mathcal L_T$ with symbol $T(p) = \sum_{i=1}^d |p_i|^{2s} = \| p \|_{2s}^{2s}$ and is the generator of a Levy-process killed upon exiting $\Omega$. In \cite{Hat13} it is proved that \eqref{weyl} holds for $\mathcal L_{2s}$ but only if $\Omega$ is a hypercube. Theorem \ref{thm:weyl} shows that this can be generalized to arbitrary open sets $\Omega$ of finite volume.

Indeed, we note that (see \cite{Hat13})
$$
V_{\|p\|_{2s}^{2s}} = \int_{\R^d} \lk \| p \|_ {2s}^{2s} - 1 \rk_-^0 dp = \frac{(2\Gamma(1+1/2s))^d}{ \Gamma(1+d/2s)} \, .
$$
Moreover, observe that for $0 < s \leq 1$ and all $r \in [0,1]$, by concavity, 
$$
r^{2s} + 1 - \frac 12 \lk (1+r^2+2r)^s + (1+r^2-2r)^s \rk \geq r^{2s}+1  - (r^2+1)^s \geq 0 \, .
$$
It follows that, for all $t \in \R$, $\lk |1+t|^{2s} + |1-t|^{2s} \rk/2  - |t|^{2s} \leq 1$. Thus, for all $p_i \in \R$ and $\eta_i \in \R$,
$$
\sum_{i=1}^d \lk  \frac 12 \lk |p_i + \eta_i|^{2s} + |p_i - \eta_i|^{2s}\rk - |p_i|^{2s}  \rk \leq \sum_{i=1}^d |\eta_i|^{2s} \, .
$$
We see that the operator $\mathcal L_{2s}$ satisfies conditions \ref{a1} and \ref{a2} and from Theorem \ref{thm:weyl} we obtain
$$
\lim_{\Lambda \to \infty} N(\Lambda) \Lambda^{-d/2s} = \frac{(2\Gamma(1+1/2s))^d}{ \Gamma(1+d/2s)} \frac{|\Omega|}{(2\pi)^d} \, .
$$
For $s = 1$ we recover the result of Weyl \cite{Wey12a} about the Laplacian.

\section{Auxiliary results}
\label{sec:aux}

Our method of proof is rather  general  and applicable in other situations as well. We follow the strategy of \cite{Fra09a}, where an operator with magnetic field is considered. This approach is based on an application of coherent states. Moreover, we rely on the following sharp spectral estimate which is the analogue of the Berezin-Li-Yau inequality for the Dirichlet Laplacian \cite{Ber72a,Lie73,LiYau83}, see also \cite{Lap97}.

\begin{theorem}
\label{thm:ber}
For $\Omega \subset \R^d$ of finite volume and all $\Lambda > 0$
$$
\tr \lk \mathcal L_T - \Lambda  \rk_- \leq   \frac{|\Omega|}{(2\pi)^d} \int_{\R^d} \lk T(p) - \Lambda \rk_- dp \, .
$$
\end{theorem}

\begin{proof}
We follow the proof given in \cite{FraGei11,FraGei13a} for the Laplacian and the fractional Laplacian.
 Let $\chi_\Omega$ denote the characteristic function of $\Omega$. Then due to the definition of $\mathcal L_T$ (note that it is defined on functions that are zero on the complement of $\Omega$)
$$
\tr \lk \mathcal L_T - \Lambda \rk_- = \tr ( \chi_\Omega ( \mathcal L_T - \Lambda ) \chi_\Omega )_- \leq \tr ( \chi_\Omega ( \mathcal L_T - \Lambda )_- \chi_\Omega ) \, , 
$$
where the second relation follows from the variational principle. Now we can write out the kernel of the operator on the right and calculate its trace. We obtain
\begin{align*}
\tr \lk \mathcal L_T - \Lambda \rk_- &\leq \frac 1{(2\pi)^d}\int_{\R^d} \left[ \chi_\Omega(x) \int_{\R^d} e^{ip\cdot(x-y)} \lk T(p) - \Lambda \rk_- dp \, \chi_\Omega(y) \right]_{x=y} dx \\
&= \frac{|\Omega|}{(2\pi )^d} \int_{\R^d}  \lk T(p)-\Lambda \rk_- dp 
\end{align*}
and the proof is complete.
\end{proof}

If the symbol $T(p)$ is homogeneous of degree $\alpha$, so that $T(p) = T_0(p)$ in assumption \ref{a1}, then we change variables $p = \mu^{1/\alpha} \xi$, $\mu > 0$, and by Fubini's theorem we get
\begin{align}
\nonumber
 \int_{\R^d} \lk T(p) - \Lambda \rk_- dp &= \int_{\R^d} \int_0^\Lambda \lk T(p) - \mu \rk_-^0 d\mu \, dp \\
 \nonumber
 &= \int_0^\Lambda \mu^{d/\alpha} \int_{\R^d} \lk \mu^{-1} T(\mu^{1/\alpha} \xi) - 1 \rk_-^0 d\xi \, d\mu \\
 \label{hom}
 &= \frac \alpha{d+\alpha} \, V_T \, \Lambda^{1+d/\alpha} \, .
\end{align}
Hence for homogeneous symbols we obtain from Theorem \ref{thm:ber}
\begin{equation}
\label{berhom}
\tr \lk \mathcal L_T - \Lambda  \rk_- \leq \frac \alpha{d+\alpha}  \,  \frac{|\Omega|}{(2\pi)^d} \, V_T  \, \Lambda^{1+d/\alpha} 
\end{equation}
for all $\Lambda > 0$.
We note that $\tr \lk \mathcal L_T - \Lambda  \rk_- = \sum_k(\Lambda - \lambda_k)_+ = \int_0^{\Lambda} N(\mu) d\mu$.
Inserting the asymptotic result from Theorem~\ref{thm:weyl} we see that \eqref{berhom} yields a sharp bound.

In fact, \eqref{berhom} is equivalent to a sharp lower bound on the sum of the eigenvalues.
Note that, for all $M \in \N$, 
$$
\sum_{k = 1}^M \lambda_k = \sup_{\Lambda > 0} \lk M \Lambda - \sum_{k \in \N} \lk \Lambda - \lambda_k \rk_+  \rk
$$
so that, by \eqref{berhom},
\begin{equation}
\label{liyau}
\sum_{k = 1}^M \lambda_k \geq \frac{d}{d+\alpha} \, \frac{(2\pi)^\alpha}{\lk |\Omega| V_T \rk^{\alpha/d}} \, M^{1+\alpha/d} \, .
\end{equation}
The bounds \eqref{berhom} and \eqref{liyau} are generalizations of the Berezin-Li-Yau inequality for the Dirichlet Laplacian.
For the operator $\mathcal L_{2s}$ we obtain the bound
$$
\sum_{k=1}^M \mathcal \lambda_k \geq \frac d{d+2s} \frac{(2\pi)^{2s} \Gamma(1+d/2s)^{2s/d}}{(2\Gamma(1+1/2s))^{2s}} |\Omega|^{2s/d} M^{1+2s/d} 
$$
that was derived in \cite{Hat13} using the arguments from \cite{LiYau83}.

We also need the following well-known lemma, see for example \cite[Lemma 17.1]{Kor04}, which is a weak, elementary version of a Tauberian theorem. For completeness we include the short proof.

\begin{lemma}
\label{lem:tauber}
Let $(\xi_k)_{k \in \N}$ be a non-decreasing sequence of positive numbers that tends to infinity. Assume that there are finite constants $A > 0$ and $a > 0$ such that 
$$
\lim_{\Lambda \to \infty} \Lambda^{-a-1} \sum_{k \in \N} \lk \Lambda - \xi_k \rk_+ = A \, .
$$
Then
$$
\lim_{\Lambda \to \infty} \Lambda^{-a} \sum_{k \in \N} \lk \Lambda -\xi_k \rk_+^0 = (a+1) A \, .
$$
\end{lemma}

\begin{proof}
Let us introduce the notation  $S(\Lambda) = \sum_k \lk \Lambda - \xi_k \rk_+$ and $N(\Lambda) = \sum_k \lk \Lambda - \xi_k \rk_+^0$.
For any $h > 0$ and $\Lambda > 0$, $k \in \N$ we have $(\Lambda + h - \xi_k)_+ - (\Lambda - \xi_k)_+ \geq h (\Lambda-\xi_k)_+^0$, thus
\begin{equation}
\label{tauberproof}
S(\Lambda + h ) - S(\Lambda) \geq h N(\Lambda) \, .
\end{equation}
By assumption, for fixed $0 < \epsilon \leq 1$, we find $\Lambda$ large enough such that $|\Lambda^{-a-1} S(\Lambda) - A| \leq \epsilon$. Hence, for $0 < h \leq \Lambda$ relation \eqref{tauberproof} implies
\begin{align*}
N(\Lambda) & \leq \frac 1h \lk A \lk \lk \Lambda + h \rk^{a+1} - \Lambda^{a+1} \rk + \epsilon \lk (\Lambda+h)^{a+1} + \Lambda^{a+1} \rk \rk \\
& \leq  (a+1) A \Lambda^a + C \Lambda^{a-1} h + \frac \epsilon h \lk 2^{a+1} +1 \rk \Lambda^{a+1} 
\end{align*}
with a constant $C > 0$ depending only on $A$ and $a$. Now we choose $h = \sqrt \epsilon \Lambda$ and obtain
$$
\Lambda^{-a} N(\Lambda) \leq  (a+1) A  + \sqrt \epsilon \lk C + \lk 2^{a+1} +1 \rk \rk \, .
$$
Since $0 < \epsilon \leq 1$ was arbitrary this completes the proof of the upper bound. The lower bound follows similarly using the fact that  $(\Lambda - \xi_k)_+ -  (\Lambda - h - \xi_k)_+ \leq h (\Lambda-\xi_k)_+^0$ for $h > 0$.  
\end{proof}

\section{Proof of the main result}

With Theorem \ref{thm:ber} and Lemma \ref{lem:tauber} at hand we can now give the proof of Theorem \ref{thm:weyl}.
By Lemma \ref{lem:tauber} and the fact that $\tr \lk \mathcal L_T - \Lambda \rk_- = \sum_k (\Lambda - \lambda_k)_+$ it suffices to prove the asymptotics
\begin{equation}
\label{tauber}
\tr \lk \mathcal L_T - \Lambda \rk_- =\frac{\alpha}{\alpha +d} \, \frac{|\Omega|}{(2\pi)^d} V_T \, \Lambda^{1+d/\alpha} \lk 1 + o(1) \rk \, , \qquad \Lambda \to \infty \, .
\end{equation}

The upper bound follows from Theorem~\ref{thm:ber}.  Indeed, writing $p = \Lambda^{1/\alpha} \xi$ we get
$$
\tr \lk \mathcal L_T - \Lambda \rk_-  \leq \frac{|\Omega|}{(2 \pi)^d} \int_{\R^d} \lk T(p) - \Lambda \rk_- dp =  \frac{|\Omega|}{(2 \pi)^d} \int_{\R^d} \lk \Lambda^{-1} T ( \Lambda^{1/\alpha}  \xi ) - 1\rk_- d\xi \,  \Lambda^{1+d/\alpha} \, .
$$
By assumption \ref{a1}, dominated convergence shows that
$$
\tr \lk \mathcal L_T - \Lambda \rk_-  \leq  \frac{|\Omega|}{(2 \pi)^d} \int_{\R^d} \lk T_0(\xi) - 1\rk_- d\xi \,  \Lambda^{1+d/\alpha} \lk 1 + o(1) \rk
$$
as $\Lambda \to \infty$. Hence, the upper bound in \eqref{tauber} follows from the identity
\begin{equation}
\label{tvol}
 \int_{\R^d} \lk T_0(\xi) - 1\rk_- d\xi = \frac{\alpha}{\alpha + d} V_T
\end{equation}
which is derived in the same way as equation \eqref{hom} in Section \ref{sec:aux}.

To prove the lower bound fix $\delta > 0$ and put $\Omega_\delta = \{ x \in \Omega  : \mbox{dist}(x,\R^d \setminus \Omega) > \delta \}$. By dominated convergence, $|\Omega_\delta| \to |\Omega|$ as $\delta \to 0$, hence it suffices to show the lower bound in  \eqref{tauber} with $\Omega$ replaced by $\Omega_\delta$.

Let $g \in C_0^\infty(\R^d)$ be a real-valued, $L^2$-normalized function with support in $\{x \in \R^d \, : \, |x| \leq \delta/2\}$. For $p \in \R^d$ and $q \in \Omega_\delta$ we introduce the coherent states $F_{p,q}(x) = e^{ip\cdot x} g(x-q)$. Then the properties of coherent states (see, e.g, \cite[Thm. 12.8]{LieLos01}) imply
$$
\tr \lk \mathcal L_T - \Lambda \rk_- \geq \frac1{(2\pi)^d} \iint_{\R^d \times \Omega_\delta} \ls F_{p,q}, \lk \mathcal L_T - \Lambda \rk_- F_{p,q} \rs dp \, dq \, .
$$
Note that the map  $t \mapsto (t- \Lambda)_-$ is convex and that $\| F_{p,q} \|_{L^2(\R^d)} = 1$. Thus we can apply Jensen's inequality to the spectral measure of $\mathcal L_T$ and obtain
$$
\tr \lk \mathcal L_T - \Lambda \rk_- \geq \frac1{(2\pi)^d} \iint_{\R^d \times \Omega_\delta} \lk \left<  F_{p,q},  \mathcal L_T F_{p,q} \right> - \Lambda \rk_-  dp \, dq \, .
$$

We claim that there is a constant $C > 0$ (independent of $p$, $q$, and $\Lambda$) such that
\begin{equation}
\label{cohest}
\ls F_{p,q},  \mathcal L_T F_{p,q} \rs \leq T(p) + C \, .
\end{equation}
Then, after inserting this estimate into the bound above we can integrate over $q \in \Omega_\delta$ and find
$$
\tr \lk \mathcal L_T - \Lambda \rk_-  \geq \frac{|\Omega_\delta|}{(2\pi)^d} \int_{\R^d} \lk T(p) + C - \Lambda \rk_- dp \, .
$$
In the same way as above the relations
\begin{align*}
\int_{\R^d} \lk T(p) + C - \Lambda \rk_- dp &= \int_{\R^d} \lk \Lambda^{-1} T(\Lambda^{1/\alpha} \xi) + \Lambda^{-1} C - 1 \rk_- d\xi \, \Lambda^{1+d/\alpha} \\
&= \int_{\R^d}  \lk T_0(\xi) - 1 \rk_- d\xi \,  \Lambda^{1+d/\alpha} \lk 1 + o(1) \rk \\
&= \frac{\alpha}{d+\alpha} V_T \Lambda^{1+d/\alpha} \lk 1 + o(1) \rk
\end{align*}
follow from dominated convergence and the homogeneity of $T_0$ by assumption \ref{a1}. 
Thus we conclude
$$
\liminf_{\Lambda \to \infty} \Lambda^{-1-d/\alpha} \, \tr \lk \mathcal L_T - \Lambda \rk_- \geq \frac{\alpha}{d+\alpha} \, \frac{|\Omega_\delta|}{(2\pi)^d} \, V_T  \, .
$$
This is the required lower bound of \eqref{tauber}. Hence,  to complete the proof of Theorem \ref{thm:weyl} it remains to establish \eqref{cohest}.

By definition of $\mathcal L_T$ and $F_{p,q}$ we have
$$
\ls F_{p,q},  \mathcal L_T F_{p,q} \rs = \frac 1{(2\pi)^d} \iiint e^{i(p-\xi) \cdot (x-y)} g_q(x)g_q(y) T(\xi) \, dx  \, dy \, d\xi \, ,
$$
where we write $g_q(x) = g(x-q)$ and all integrals are taken over $\R^d$. We insert the identity $g_q(x)g_q(y) =   (g_q(x)^2 + g_q(y)^2- (g_q(x)-g_q(y))^2  )/2$ and use the symmetry of the integral above to obtain
\begin{align}
\nonumber
\ls F_{p,q},  \mathcal L_T F_{p,q} \rs &=\iiint e^{i(p-\xi) \cdot (x-y)}  \lk g_q(x)^2 - \frac 12 \lk g_q(x) - g_q(y) \rk^2 \rk T(\xi) \frac{ dx  dy d\xi}{(2\pi)^d} \\
\label{cohest2}
&= T(p) - \iiint e^{i(p-\xi) \cdot (x-y)}  \lk g_q(x) - g_q(y) \rk^2  T(\xi) \frac{ dx  dy d\xi}{2(2\pi)^d} \, ,
\end{align}
where in the second step we used the fact that $\int g(x-q)^2 dx = 1$ for all $q \in \R^d$.

To estimate the contribution of the second summand we substitute $y = x -z$ and apply Plancherel's Theorem to get
\begin{align*}
 &\iiint e^{i(p-\xi) \cdot z}  \lk g_q(x) - g_q(x-z) \rk^2 T(\xi) \frac{ dx  dz d\xi}{2(2\pi)^d} \\ 
 &=  \iiint e^{i(p-\xi) \cdot z} | \hat g_q(\eta) |^2 \left| 1- e^{-iz\cdot \eta}\right|^2 T(\xi) \frac{ d\eta  dz d\xi}{2(2\pi)^d} \, .
 \end{align*} 
We note that $|\hat g_q(\eta)|^2 = |\hat g(\eta)|^2$ is independent of $q$. Moreover, we write $| 1- e^{-iz\cdot \eta} |^2 = 2 - e^{iz \cdot \eta} - e^{-iz \cdot \eta}$ and perform the integration in $z$ and $\xi$. We find
\begin{align*}
 &\iiint e^{i(p-\xi) \cdot (x-y)}  \lk g_q(x) - g_q(y) \rk^2  T(\xi) \frac{ dx  dy d\xi}{2(2\pi)^d} \\
 &= \int \lk T(p) - \frac 12 \lk T(p+\eta) + T(p - \eta) \rk \rk |\hat g(\eta)|^2 d\eta \, .
\end{align*}
By assumption \ref{a2} the right-hand side is bounded below by $- C_0 \int (1+|\eta|)^N  |\hat g(\eta)|^2 d\eta \geq - C$. 
Combining this estimate with \eqref{cohest2} yields \eqref{cohest} and completes the proof of Theorem \ref{thm:weyl}.

\subsection*{Acknowledgments} It is a pleasure to thank Rupert L. Frank and Anna Vershynina for helpful comments. Financial support from DFG grant GE 2369/1-1 and NSF grant PHY-1122309 is gratefully acknowledged.

\end{document}